\documentclass{article}
\usepackage{one_inc}
\title{On the Injective Norm of Sums of Random Tensors and the Moments of Gaussian Chaoses}
\author{Ishaq Aden-Ali\thanks{Department of Electrical Engineering and Computer Science, UC Berkeley. Email: adenali@berkeley.edu}}

\begin{document}
\maketitle
\begin{abstract}
We prove an upper bound on the expected $\ell_p$ injective norm of sums of subgaussian random tensors.
Our proof is simple and does not rely on any explicit geometric or chaining arguments.
Instead, it follows from a simple application of the \emph{PAC-Bayesian lemma}, a tool that has proven effective at controlling the suprema of certain ``smooth'' empirical processes in recent years.
Our bound strictly improves a very recent result of Bandeira, Gopi, Jiang, Lucca, and Rothvoss.
In the Euclidean case ($p=2$), our bound sharpens a result of Lata{\l}a that was central to proving his estimates on the moments of Gaussian chaoses.
As a consequence, we obtain an elementary proof of this fundamental result. 
\end{abstract}
\section{Introduction}
Many problems in probability theory, computer science, and statistics naturally reduce to bounding the norm of a random tensor of the form
\begin{equation}\label{eq:random_tensor_model}
T = \sum_{k=1}^n \xi_k T_k  ,  
\end{equation}
where $( \xi_k )$ is a sequence of independent zero mean subgaussian random variables\footnote{Recall that the random variable $\xi \in \bR$ is subgaussian if  $\E\exp(\lambda(\xi - \E \xi)) \le \exp(\frac{\lambda^2}{2})$ for all $\lambda > 0$. } and each $T_k $ is a deterministic order $r$ tensor living in the tensor space $\mathbb{R}^{d_1} \otimes \dots \otimes \mathbb{R}^{d_r}$. 
We will refer to tensors of the form \eqref{eq:random_tensor_model} as \emph{subgaussian random tensors} throughout.
In this paper, our goal is to understand the typical behaviour of the $\ell_p$ \emph{injective norm} of $T$, which is defined as
\[
\pnorm{T} = \sup_{x_1 \in B_p^{d_1}, \dots,x_r \in B_p^{d_r}} T(x_1, \dots, x_r),
\]
where $B_p^d  \coloneqq \{x \in \mathbb{R}^d : \|x\|_p \le 1\}$ is the $d$-dimensional $\ell_p$ ball.
In the special case of Euclidean ($p=2$) matrices ($r=2$), it is not hard to see that the injective norm corresponds to the usual operator norm.
Our understanding of the typical behaviour of the operator norm of random matrices is quite deep; see the books \cite{tao2012topics,tropp15,Vershynin2016HDP} and references therein. 
Unfortunately, this understanding \emph{does not} extend to higher order tensors.
Indeed, much of the existing machinery developed in the matrix setting fails to generalize, perhaps chiefly due to our inability to extend crucial matrix trace inequalities.
Because of this, we do not have many bounds on the injective norm for general random tensor models like \eqref{eq:random_tensor_model}.
Nonetheless, one can still try and derive such bounds as we do in this paper.

Before we state our result, we first need to introduce some cumbersome tensor notation in order to define the relevant parameters in our bounds.
For a tensor $A \in \tensorspace$ and vectors $x_1 \in \mathbb{R}^{d_1}, \dots, x_{\ell} \in \mathbb{R}^{d_\ell}$ we can define a contraction of $A$ along its first $\ell$ indices using $x_1, \dots, x_\ell$, denoted by  $A(x_1, \dots, x_{\ell}, \cdot, \dots, \cdot)$, as the order $r-\ell$ tensor in $\bR^{d_{\ell+1}} \otimes \dots \otimes \bR^{d_r}$ with entries
\[
A(x_1, \dots, x_{\ell}, \cdot, \dots, \cdot)_{i_1, \dots, i_{r-\ell}} = \sum_{j_1 \in [d_1], \dots, j_\ell \in [d_\ell]} A_{j_1, \dots, j_\ell, i_1, \dots, i_{r-\ell}} \prod_{t = 1}^{\ell} x_{j_t}.
\]
We also need to consider contractions of $A$ via arbitrary indices.
For example, if $r=3$ then the order $2$ tensor $A(\cdot, x_2, \cdot)$ is defined in the obvious way in our context.
Unfortunately, considering arbitrary inputs to contract on for high order tensors requires additional notation to make things readable.
For a (possibly empty) subset $I \subseteq [r]$, let $\pi_I$ be the permutation that maps the elements $t_1 < \dots < t_{|I|} \in I $ to $\{1, \dots, |I|\}$ in the same relative order and maps the elements $t_{|I|+1}< \dots < t_{r} \in I^c$ to $\{|I|+1, \dots, r\}$ in the same relative order.
\footnote{In other words, this is the permutation that shifts down the ``active'' indices in $I$ to $\{1, \dots, |I|\}$ and shifts up the ``inactive'' indices in $I^c$ to $\{|I|+1, \dots, r\}$.}
We write $\pi_I(i)$ to denote the $i$th element of the permutation.
For a tensor $A \in \tensorspace$ and a subset of indices $I \subseteq [r]$, we define the permuted tensor $A^{\pi_I}$ in the space $\bR^{d_{\pi_I(1)}} \otimes \dots \otimes \bR^{d_{\pi_I(r)}}$ with entries $A^{\pi_I}_{i_1, \dots, i_r} = A_{i_{\pi_I(1)}, \dots, i_{\pi_I(r)}}$.
To see how this notation is useful, consider our previous example of an order $3$ tensor $A \in \bR^{d_1} \otimes \bR^{d_2} \otimes \bR^{d_3}$, vectors $x_1 \in \bR^{d_1}, x_2 \in \bR^{d_2}, x_3 \in \bR^{d_3}$, and the set $I = \{ 2 \}$.
It is easy to see that the order $2$ tensor $A^{\pi_I}(x_{\pi_I(1)}, \cdot, \cdot)$ is equivalent to the tensor $A(\cdot, x_{2}, \cdot)$ (up to a permutation).
Finally, we define the \emph{Frobenius norm} of a tensor $A \in \tensorspace$, denoted $\|A\|_{F}$, to be the square root of the sum of squared entries of $A$.
We are now ready to define the variance parameters.
\begin{definition}(Variance parameters.)\label{def:astd} Let $T \in \tensorspace$ be a subgaussian random tensor. For any $p \ge 2$ and $\ell \in \{0,1,\dots, r\}$ we define the variance parameter
\[
\astd{\ell}^2 = \astd{\ell}^2(T) = \sup_{x_1 
\in B_p^{d_1}, \dots, x_r \in B_p^{d_r}}  \sum_{\substack{I \subseteq [r] \\ |I| = r - \ell}}  \sum_{k=1}^n \|T_k^{\pi_I}(x_{\pi_I(1)}, \dots, x_{\pi_I(r-\ell)}, \cdot, \dots, \cdot) \|_{F}^2.
\]
\end{definition}
Our main result is the following clean bound on the expected $\ell_p$ injective norm of subgaussian random tensors in terms of these variance parameters.
\begin{theorem}\label{thm:main}
Let $T \in \tensorspace$ be a subgaussian random tensor.
For any $p \ge 2$ and $\beta > 0$ we have
\begin{equation}
    \E \pnorm{T} \le \sqrt{\left(\sum_{\ell =0}^r \beta^{1-\ell}{\astd{\ell}^2}\right)\sum_{t=1}^r d_{t}^{1-\frac{2}{p}} }.
\end{equation}
\end{theorem}
Approximately optimizing the parameter $\beta$ above yields the following corollary.
\begin{corollary}\label{cor:interpretable}
Let $T \in \tensorspace$ be a subgaussian random tensor.
For any $p \ge 2$ we have
\begin{equation}
   \E \pnorm{T} \le \sqrt{\left(\astd{1}^2 + r\max_{2 \le \ell \le r}\astd{\ell}^{\frac{2}{\ell}}\astd{0}^{\frac{2\ell-2}{\ell}}\right)\sum_{t=1}^r d_{t}^{1-\frac{2}{p}}}.
\end{equation}
\end{corollary}

The very recent work of Bandiera, Gopi, Jiang, Lucca, and Rothvoss \cite{bandeira2024tensor} provides a bound on the expected $\ell_p$ injective norm of subgaussian random tensors when each $T_k$ is a symmetric tensor.\footnote{The symmetry assumption is without loss of significant generality as they show that their bounds apply to general square tensors up to a loss in constant factors that only depend on the order $r$.} 
Our result is directly comparable to theirs. \cref{prop:symmetric_variance_comparison} shows that for symmetric tensors the variance parameters we consider can be bounded by scaled versions of the variance parameters appearing in \cite{bandeira2024tensor} (where the scaling only depends on the order $r$).
The following corollary phrased in terms of their variance parameters $(\std{\ell}^2)$ (defined in \cref{sec:variances}) is a direct consequence of this simple observation.
\begin{corollary}\label{cor:sym}
Let $T \in (\mathbb{R}^d)^{\otimes r}$ be a symmetric subgaussian random tensor.
For any $p \ge 2$ we have
\begin{equation}
    d^{\frac{1}{p} - \frac{1}{2}}\E \pnorm{T} \lesssim_{r} \std{1} + \max_{2 \le \ell \le r}\std{\ell}^{\frac{1}{\ell}}\std{0}^{\frac{\ell-1}{\ell}}.
\end{equation}
\end{corollary}
This bound improves over the main result of~\cite{bandeira2024tensor} by removing a $\log d$ from the first term.
The proof of the main result of \cite{bandeira2024tensor} begins by reformulating the problem as one of controlling the expected supremum of a certain Gaussian process.
A classic result in the theory of Gaussian processes implies that this supremum can be bounded by Dudley's entropy integral (see \cite[Chapter 1]{talagrand2022}).
The main technical contribution of \cite{bandeira2024tensor} was deriving sufficiently sharp estimates on the relevant covering numbers which required understanding the underlying (complicated) geometry of the Gaussian process.

In this work, we completely avoid using any explicit geometric or chaining arguments.
Instead, we control the relevant (sub)gaussian process by appealing to the so-called \emph{PAC-Bayesian lemma}.
At a high level, the PAC-Bayesian lemma provides a way to control the supremum of certain ``smoothed" processes, including the multilinear processes we study in this paper.
We defer a formal statement and discussion of this lemma to~\cref{sec:pac_bayes}.

Taking a step back, we should ask ourselves: How sharp are these bounds? 
A natural setting in which to consider this question is that of Euclidean matrices.
It is easy to see that \cref{thm:main} implies the bound
\[
\E\|T\|_{\mathrm{op}} \lesssim \max\left\{\left\|\sum_{k=1}^n T_k^{\top}T_k\right\|_{\mathrm{op}}^{\frac{1}{2}},\left\|\sum_{k=1}^n T_kT_k^{\top}\right\|_{\mathrm{op}}^{\frac{1}{2}}\right\} + \left(\sup_{x \in B_2^{d_1}, y \in B_2^{d_2}}\sum_{k=1}^n (x^{\top} T_k y)^2\right)^{\frac{1}{4}}\left(\sum_{k=1}^n \|T_k\|_F^2\right)^{\frac{1}{4}},
\]
where $\|\cdot\|_{\mathrm{op}}$ denotes the operator norm of a matrix.
The first term on the right-hand side of the above is precisely the term that appears in the (non-commutative) matrix Khintchine inequality \cite{LustPiquard1986, LustPiquardPisier1991}, but without the usual $\sqrt{\log d}$ factor that is necessary in general.
There is no contradiction, however, since the sum of squared Frobenius norms in the second term generally yields \emph{much larger} upper bounds than the matrix Khintchine inequality in regimes where the multiplicative $\sqrt{\log d}$ factor is necessary.
For random tensors of order $r \ge 3$, the very recent work of Boedihardjo~\cite{boedihardjo2024injective} provides a bound on the expected $\ell_2$ injective norm of Gaussian random tensors in the independent entry model which is a special case of \eqref{eq:random_tensor_model}.
His bound is sharper than the bound implied by \cref{thm:main} in a wide range of settings, analogous to the suboptimal regimes previously mentioned in the matrix case.

While it is unfortunate that our bounds do not recover the matrix Khintchine inequality, this is somewhat expected since, as we will see, our proof is agnostic to the underlying matrix structure.
Loosely speaking, the proof of our main result is ``volumetric'' in nature.
We should briefly remark that in the Euclidean matrix setting, there is a way to use a very different kind of PAC-Bayesian argument that utilizes the underlying matrix structure \cite{adenali202X}. 
This approach ever-so-slightly generalizes Tropp's celebrated matrix moment generating function (MGF) approach \cite{tropp2012user,tropp15} and is able to recover the matrix Khintchine inequality (among other things). 
Still, our \emph{dimension-free} $\ell_2$ injective norm bounds for $r \ge 3$ are useful in applications where the Frobenius norm terms are (1) not too large (e.g.\ due to sparsity) or are (2) inherently necessary due to application-specific considerations.
A few applications of the form (1) are investigated in \cite{bandeira2024tensor}.
An application of the form (2) that we will consider is that of estimating the moments of \emph{Gaussian chaoses}.
\subsection{Moments of Gaussian Chaoses}
An order $r$ (decoupled) Gaussian chaos is a random variable of the following form:
\[
\sum_{i_1, \dots, i_r = 1} ^d A_{i_1, \dots, i_r}\prod_{t=1}^r (g_t)_{i_t}
\]
where $A \in (\bR^d)^{\otimes r}$ is a deterministic tensor and $g_1, \dots, g_r \in \bR^d$ are independent standard Gaussian random vectors. 
Recall that the $p$th moment of a random variable $X \in \mathbb{R}$ is defined as $\|X\|_p = (\E |X|^p)^{\frac{1}{p}}$.
A fundamental result in probability theory due to Lata{\l}a~\cite{latala2006estimates} is a sharp upper bound on the moments of decoupled Gaussian chaoses of all orders.\footnote{When the tensor $A$ contains certain symmetries, a classic decoupling result \cite{delaPena1995} allows one to apply Lata{\l}a's result to \emph{coupled} Gaussian Choases $\sum_{i_1, \dots, i_r \in [d]} A_{i_1, \dots, i_r}\prod_{t=1}^r g_{i_t}$ that are defined by a single Gaussian vector $g \in \bR^d$.}
This generalizes a well known result due to Hanson and Wright for the case $r=2$ \cite{HansonWright1971}.
To formally state Lata{\l}a's result, we unfortunately need some more notation.
Let $S(r)$ be the set of all non-empty partitions of $[r] \coloneqq \{1, \dots, r\}$.
Given a tensor $A \in (\bR^d)^{\otimes r}$ and a partition $\mc{P} = \{I_1, \dots, I_m\} \in S(r)$, let $A^{\mc{P}}$ be the tensor $A$ when viewed as an element of the space $\bR^{d^{|I_1|}} \otimes \dots \otimes \bR^{d^{|I_m|}}$ in the natural way: the $j$th index of $A^{\mc{P}}$ is defined by a tuple $(i_t)_{t \in I_j}$ which represents a subset of indices of the original tensor $A \in (\bR^d)^{\otimes r}$.
For example, if $A \in (\bR^d)^{\otimes 3}$ and $\mc{P} = \{ \{1\}, \{2,3\} \}$, the $(i_1, (i_2, i_3))$th entry of $A^{\mc{P}} \in \bR^d \otimes \bR^{d^2}$ is the $(i_1, i_2, i_3)$th entry of $A$.
The following theorem is Lata{\l}a's celebrated result~\cite{latala2006estimates}.
\begin{theorem}\label{thm:moment_gaussian_chaos}
There is a constant $C_r$ that only depends on the order $r$ such that for every $p \ge 2$ we have 
\[
\left\|\sum_{i_1, \dots, i_r \in [d]}A_{i_1, \dots, i_r} \prod_{t=1}^r (g_t)_{i_t}\right\|_p \le C_r \sum_{\mc{P} \in S(r)} p^{\frac{|\mc{P}|}{2}}\normtwo{A^{\mc{P}}}.
\]
\end{theorem}
Lata{\l}a also proved a lower bound that matches the above up to constants that only depend on $r$. 
The centerpiece of Lata{\l}a's proof of \cref{thm:moment_gaussian_chaos} is a bound on the expected $\ell_2$ injective norm of Gaussian random tensors following the model \eqref{eq:random_tensor_model}.
His bound is much larger than the one implied by \cref{thm:main}.
The proof of his estimate is based on bounding the relevant Gaussian process via a somewhat ad-hoc chaining argument that seems to be distinct from both Dudley's integral and Talagrand's celebrated generic chaining \cite{talagrand2022}.
Moreover, his proof is very technical and is quite difficult to follow.
We refer the curious reader to Talagrand's book \cite[Section 15.2]{talagrand2022} for a textbook level treatment of Lata{\l}a's proof for the case $r=3$ that seems to contain all the ideas needed for the general result.

Since our main result recovers Lata{\l}a’s estimate for the expected $\ell_2$ injective norm of Gaussian tensors via a simple argument, we also obtain an elementary proof of~\cref{thm:moment_gaussian_chaos}. 
The hardest part of the proof (and this paper in general) is digesting the tensor notation. 
We include the proof in \cref{sec:moments} for completeness.

\section{Preliminaries}\label{sec:prelim}
We begin by introducing some notation that will be used throughout.
We use $C$ and $C'$ to denote universal constants whose value may change from line to line. 
We use $C_b$ to represent a constant that depends only on the number $b$ and such a $b$ will usually be a relevant parameter of the problem.
We write $f \lesssim g$ if there exists a universal constant $C > 0 $ such that $f \le C g$.
We also write $f \lesssim_{b} g $ if $f \le C_b g$.
For a positive integer $a$ we define the set $[a] = \{1, \dots, a\}$.
We write $x_i$ to denote the $i$th entry of a vector $x \in \bR^d$.
To avoid using superscripts, we will slightly abuse notation by also using subscripts to index a collection of vectors $x_1 \in \bR^{d_1} \dots x_r \in \bR^{d_r}$.
In this case we will write $(x_k)_i$ to denote the $i$th entry of the $k$th vector.
We also generalize this notation to tensors, e.g.\ $(A_k)_{i_1, \dots, i_r}$ denotes the $(i_1, \dots, i_r)$th entry of the $k$th tensor.
We further overload subscripts by denoting the $d \times d$ identity matrix by $I_d$.
We write $\mc{N}(\mu, \Sigma)$ for the Gaussian measure with mean $\mu$ and covariance $\Sigma$.

The tensor space $\bR^{d_1} \otimes \dots \otimes \bR^{d_r}$ consists of all multilinear maps from $\bR^{d_1} \times \dots \times \bR^{d_r}$ to $\bR$ where each $A \in \tensorspace$ is determined by the formula 
\[
A(x_1, \dots, x_r) = \sum_{i_1 \in [d_1], \dots, i_r \in [d_r]} A_{i_1, \dots, i_r}\prod_{t=1}^d (x_t)_{i_t}.
\]

We use $\E$ to denote the expectation of a random variable.
We write $X \sim \mu$ to denote a random variable $X$ distributed according to the probability measure $\mu$.
We will sometimes write $\E_{X \sim \mu}$ to emphasize that $X$ is distributed according to $\mu$. 
We will also sometimes write $\E_{X}$ without specifying the measure $\mu$.
For two probability measures $\mu$ and $\nu$ defined on a (measurable) space $\mc{X}$, we write $\mu \ll \nu$ if $\mu$ is absolutely continuous with respect to $\nu$. For $\mu \ll \nu$, the Kullback-Leibeler divergence between $\mu$ and $\nu$ is defined as $\KL(\mu \| \nu) = \int_{\mc{X}} \log\left(\frac{\mathrm{d}\mu}{\mathrm{d}\nu}\right) \mathrm{d}\mu$.
Given two probability measure $\mu$ and $\nu$ defined over the spaces $\mc{X}$ and $\mc{Y}$, we use $\mu \otimes \nu$ to denote the product measure that they define on the space $\mc{X} \times \mc{Y}$.

\subsection{Relationship with existing variance parameters}\label{sec:variances}
In this subsection we discuss the relationship between our variance parameters and other parameters that have already appeared in the literature.
In \cite{bandeira2024tensor}, the following variance parameters were defined for symmetric subgaussian random tensors $T \in (\bR^d)^{\otimes r}$ for $p \ge 2$ and $\ell \in \{0, 1,\dots, r\}$:
\[
\std{\ell}^2 = \sup_{x_1 \dots, x_{r-\ell} \in B_p^d} \sum_{k=1}^n\|T_k(x_1, \dots, x_{r-\ell}, \cdot, \dots, \cdot)\|_{F}^2.
\]
It is not hard to see that we can relate our variance parameters to these variance parameters.
\begin{proposition}\label{prop:symmetric_variance_comparison}
Let $T \in (\bR^d)^{\otimes r}$ be a symmetric subgaussian random tensor. 
For any $p \ge 2$ and $\ell \in \{0, 1, \dots, r\}$ we have
\[
\astd{\ell}^2 \le \binom{r}{r-\ell}\std{\ell}^2
\]  
\end{proposition}
\begin{proof}
Using the symmetry of $T_k$ we have
\begin{align*}
\astd{\ell}^2 &\le \sum_{\substack{I \subseteq [r] \\ |I| = r - \ell}} \sup_{x_1, \dots, x_r \in B_p^{d}}  \sum_{k=1}^n  \|T_k^{\pi_I}(x_{\pi_I(1)}, \dots, x_{\pi_I(r-\ell)}, \cdot, \dots, \cdot) \|_{F}^2\\
&= \sum_{\substack{I \subseteq [r] \\ |I| = r - \ell}} \sup_{x_1, \dots, x_{r-\ell} \in B_p^d  } \sum_{k=1}^n  \|T_k(x_{1}, \dots, x_{r-\ell}, \cdot, \dots, \cdot) \|_{F}^2 = \binom{r}{r-\ell} \std{\ell}^2,
\end{align*}
as claimed.
\end{proof}

The relevant variance parameters that are implicit in Lata{\l}a's work are based on the order $r+1$ tensor $A = A(\{T_k\}_{k=1}^n)  \in (\bR^d)^{\otimes r} \otimes \bR^{n}$ that is defined by stacking the tensors $T_1, T_2, \dots , T_n$ along a new index $i_{r+1}$.
For the the sake of proving \cref{thm:moment_gaussian_chaos},
it will be useful to consider the relationship between the variance parameters when we consider subgaussian random tensors $T \in (\bR^d)^{\otimes r-1}$ ``re-indexed'' by some partition $\mc{P} \in S(r-1)$.
The following observation relates our variance parameters to the variance parameters implicitly defined in Lata{\l}a's work.
\begin{proposition}\label{prop:partition_norm}
Let $T \in (\bR^d)^{\otimes r-1}$ be a subgaussian random tensor.
For any partition $\mc{P} \in S(r-1)$ and $\ell \in \{0, 1, \dots, |\mc{P}|\}$ we have
\[
\astdtwo{\ell}^2\left(T^{\mc{P}}\right) \le \sum_{\substack{\mc{Q} \subseteq S(r) \\ |\mc{Q}| = |\mc{P}|- \ell+1}}\normtwo{A^{\mc{Q}}}^2.
\]
\end{proposition}
\begin{proof}
Let $m$ be the size of the partition $\mc{P} = \{I_1, \dots, I_m\}$ and let $d_j = d^{|I_j|}$. 
Using the simple identity $\sum_{i=1}^N a_i^2 = \sup_{x \in B_2^N} \left(\sum_{i=1}^N a_i x_i\right)^2 $ we get
\begin{align*}
\astdtwo{\ell}^2\left(T^{\mc{P}}\right) &\le \sum_{\substack{I \subseteq [m] \\ |I| = m - \ell}} \sup_{x_1 \in B_2^{d_1}, \dots, x_m \in B_2^{d_m}}  \sum_{k=1}^n  \left\|\left(T_k^{\mc{P}}\right)^{\pi_I}(x_{\pi_I(1)}, \dots, x_{\pi_I(m-\ell)}, \cdot, \dots, \cdot) \right\|_{F}^2\\
&= \sum_{\substack{I \subseteq [m] \\ |I| = m - \ell}} \sup_{x_1 \in B_2^{d_1}, \dots, x_m \in B_2^{d_m}}  \sum_{k=1}^n  \sum_{\substack{i_{t'} \in [d_{t'}] \\ t' \in I^c}} \left(\sum_{\substack{i_t \in [d_t] \\ t \in I}} \left(T_k^{\mc{P}}\right)_{i_1, \dots, i_m}\prod_{t \in I} (x_t)_{i_t}\right)^2\\
&= \sum_{\substack{\mc{Q} \subseteq \mc{P} \\ |\mc{Q}| = m - \ell}}\normtwo{A^{\mc{Q} \cup \{i_{r}\cup \{i \in I : I \in \mc{P} \setminus \mc{Q} \}\}}}^2 \le \sum_{\substack{\mc{Q} \subseteq S(r) \\ |\mc{Q}| = m - \ell +1}}\normtwo{A^{\mc{Q}}}^2,
\end{align*}
as desired.
\end{proof}

\subsection{The PAC-Bayesian lemma}\label{sec:pac_bayes}
The following result is one incarnation of the PAC-Bayesian lemma.
It is a simple consequence of the Donsker-Varadhan variational formula.
We include a proof in~\cref{app:pac_bayes_proof} for completeness.
\begin{lemma}\label{lem:pacbayes}
Let $Z_1,\ldots,Z_n$ be independent random variables on a common measurable space $\mc{Z}$.
Let $\Theta$ (the \emph{parameter space}) be a subset of $\mathbb{R}^d$ and let $\pi$ be a probability measure (the \emph{prior}) on $\Theta$. 
Let $f_1, \dots, f_n : \mc{Z} \times \Theta \to \mathbb{R}$ be measurable functions such that $\E_{Z_i}[\exp(f_i(Z_i,\theta))] < \infty$ $\pi$-almost surely for all $i \in [n]$. 
Then
\begin{align}
    \E_{Z_1, \dots, Z_n} \left[\sup_{\rho \ll \pi} \sum^n_{i=1}\left(\E_{\theta \sim \rho}f_i(Z_i,\theta) - \E_{\theta \sim \rho}\log \left(\E_{Z_i}\exp(f_i(Z_i,\theta))\right)\right) - \KL(\rho \|\pi)\right] \le 0.
\end{align}
\end{lemma}
We call the probability measures $\rho$ that appear above \emph{posteriors}.
We should mention that there is a deviation version of the above lemma, but we will not use it here.

The terminology ``parameter space'', ``prior'', and ``posterior'' used in \cref{lem:pacbayes} are inherited from the original statistical learning context where this lemma was first used. \cref{lem:pacbayes} is at the heart of the \emph{PAC-Bayesian approach}.
The PAC-Bayesian approach \cite{shawe-taylor1997,mcallester1998,mcallester1999}  was initially developed as an attempt to blend the Probably Approximately Correct (PAC) learning model~\cite{vapnik1964class,valiant1984theory} with ideas from bayesian statistics.
See \cite{MAL-112} for more background.

It seems that it was first noticed by Audibert and Catoni that \cref{lem:pacbayes} is very useful for controlling certain empirical processes that arise in the context of non-asymptotic robust statistics~\cite{audibert2010linear, audibert2011robust, catoni2012challenging, catoni2016pac}.
In those papers, the relevant empirical processes controlled were related to certain robust statistical estimators.
Subsequent work has found this lemma useful for the non-asymptotic study of random matrices \cite{oliveira2016lower, mourtada2022exact,zhivotovskiy2024dimension}.
Most relevant to our work is Zhivotovskiy's paper \cite{zhivotovskiy2024dimension} where he derives dimension-free deviation inequalities for sums of random matrices in a wide range of settings. 
His results are obtained by applying \cref{lem:pacbayes} in a number of slick ways to problems that are superficially similar to the one we study here. 
In contrast, our result follows from applying \cref{lem:pacbayes} in the most elementary way. 

\section{Proofs of main results}
We will need the following simple but clunky identity. 
It states that the second moments of certain non-centered chaoses are determined by the contractions of their associated tensors.
\begin{proposition}\label{prop:second_moment_expanision}
Let $X_1 \in \bR^{d_1}, \dots , X_r \in \bR^{d_r}$ be independent zero mean vectors such that each $X_i$ has covariance matrix $\Sigma_i  = \beta^{-1} I_{d_i}$ and let $A \in \bR^{d_1} \otimes \dots \otimes \bR^{d_r}$ be a deterministic tensor.
For any fixed vectors $x_1 \in \bR^{d_1}, \dots , x_r \in \bR^{d_r}$ we have 
\[
\E \left(\sum_{i_1 \in [d_1], \dots, i_r \in [d_r]} A_{i_1, \dots , i_r} \prod_{t =1}^r (x_t + X_t)_{i_t} \right)^2 = \sum_{\ell=0}^r \beta^{-r+\ell} \sum_{\substack{I \subseteq [r]\\ |I| = \ell}} \|A^{\pi_I}(x_{\pi_I(1)}, \dots, x_{\pi_I(\ell)}, \cdot, \dots, \cdot )\|_{F}^2.
\]
\end{proposition}
\begin{proof}
Expanding the quadratic yields
\begin{align*}
    &\E \left(\sum_{i_1 \in [d_1], \dots, i_r \in [d_r]} A_{i_1, \dots , i_r}\prod_{t =1}^r (x_t + X_t)_{i_t} \right)^2 \\
    &= \sum_{i_1, i_{r+1} \in [d_1], \dots, i_r, i_{2r} \in [d_r]} A_{i_1, \dots , i_r}A_{i_{r+1}, \dots , i_{2r}}  \prod_{t =1}^r \E \left[(x_t + X_t)_{i_t} (x_t + X_t)_{i_{r+t}}\right]\\
    &= \sum_{i_1, i_{r+1} \in [d_1], \dots, i_r, i_{2r} \in [d_r]} A_{i_1, \dots , i_r}A_{i_{r+1}, \dots , i_{2r}}  \prod_{t =1}^r \left((x_t)_{i_t}(x_t)_{i_{r+t}} + (\Sigma_{t})_{i_{t}, i_{r+t}}\right)\\
    &= \sum_{i_1, i_{r+1} \in [d_1], \dots, i_r, i_{2r} \in [d_r]} A_{i_1, \dots , i_r}A_{i_{r+1}, \dots , i_{2r}}
    \sum_{I \subseteq [r]} \prod_{t \in I}(x_t)_{i_t}(x_t)_{i_{r+t}}\prod_{t' \in I^c} (\Sigma_{t'})_{i_{t'}, i_{r+t'}}\\
     &=\sum_{I \subseteq [r]}\; \sum_{\substack{i_{t'} = i_{r+t'} \in [d_{t'}] \\ t' \in I^c}} \; \sum_{\substack{i_t, i_{r+t} \in [d_t] \\ t \in I}} A_{i_1, \dots , i_r}A_{i_{r+1}, \dots , i_{2r}}
     \prod_{t \in I}(x_t)_{i_t}(x_t)_{i_{r+t}} \beta^{-r +|I|}\\
     &=\sum_{\ell=0}^r \beta^{-r+\ell} \sum_{\substack{I \subseteq [r]\\ |I| = \ell}} \;\sum_{\substack{i_{t'} = i_{r+t'} \in [d_{t'}] \\ t' \in I^c}} \; \sum_{\substack{i_t, i_{r+t} \in [d_t] \\ t \in I}}  A_{i_1, \dots , i_r}A_{i_{r+1}, \dots , i_{2r}}
     \prod_{t \in I}(x_t)_{i_t}(x_t)_{i_{r+t}}\\
     &=\sum_{\ell=0}^r \beta^{-r+\ell} \sum_{\substack{I \subseteq [r]\\ |I| = \ell}} \;\sum_{\substack{i_{t'} \in [d_{t'}] \\ t' \in I^c}} \; \left(\sum_{\substack{i_t \in [d_t] \\ t \in I}}  A_{i_1, \dots , i_r}
     \prod_{t \in I}(x_t)_{i_t}\right)^2\\
     &= \sum_{\ell=0}^r \beta^{-r+\ell} \sum_{\substack{I \subseteq [r]\\ |I| = \ell}} \|A^{\pi_I}(x_{\pi_I(1)}, \dots, x_{\pi_I(\ell)}, \cdot, \dots, \cdot )\|_{F}^2,
\end{align*}
as desired.
\end{proof}
We are now ready to prove the main result.
\begin{proof}[Proof of \cref{thm:main}]
We will use the PAC-Bayesian lemma (\cref{lem:pacbayes}) to bound the following quantity:
\begin{equation}\label{eq:tensor_process}
\E_{\xi_1 \dots, \xi_n}\left[ \sup_{x_1 \in B_p^{d_1}, \dots, x_r \in B_p^{d_r}} \sum_{k=1}^n \xi_k T_k(x_1, \dots, x_r) \right].
\end{equation}
To do so, we need to make some choices.
We choose our parameter space to be $\Theta = \bR^{d_1} \times \dots \times \bR^{d_r}$ and pick our prior 
\[
\pi = \mc{N}(0, \beta^{-1} I_{d_1}) \otimes \dots \otimes \mc{N}(0, \beta^{-1} I_{d_r})
\]
over $\Theta$, where the parameter $\beta > 0$ is arbitrary.
Every parameter $\theta = (\theta_1, \dots, \theta_r) \in \Theta$ has a corresponding posterior $\rho_\theta = \mc{N}(\theta_1, \beta^{-1} I_{d_1}) \otimes \dots \otimes \mc{N}(\theta_r, \beta^{-1} I_{d_r})$.
Define the set $\mc{M}_{p} = \{ \rho_\theta : \theta \in B_p^{d_1} \times \dots \times  B_p^{d_r}\}$. 
As we will see below, these probability measures play a crucial role in how we apply the PAC-Bayesian lemma.
For $k \in [n]$ define the function
\[
f_k(\xi_k, \theta) = \lambda \xi_k \sum_{i_1, \dots, i_r=1}^d(T_k)_{i_1, \dots, i_r}\prod_{t=1}^r (\theta_t)_{i_t},
\]
where the parameter $\lambda > 0$ will be set below.
We can now reformulate \cref{eq:tensor_process} as a smooth process by exploiting the linearity of expectation.
Since $T_k$ is multi-linear in $x_1, \dots, x_k$, we can re-write \cref{eq:tensor_process} as
\begin{equation}\label{eq:smooth_process}
\E_{\xi_1 \dots, \xi_n}\left[ \sup_{x_1 \in B_p^{d_1}, \dots, x_r \in B_p^{d_r}}\sum_{k=1}^n \xi_k T_k(x_1, \dots, x_k ) \right] = \frac{1}{\lambda}\E_{\xi_1 \dots, \xi_n}\left[ \sup_{\rho_\theta \in \mc{M}_p} \sum_{k=1}^n \E_{\theta \sim \rho_\theta}f_k(\xi_k, \theta)\right].
\end{equation}
Applying \cref{lem:pacbayes} with $Z_i = \xi_i$ to the smooth process on the right-hand side of \cref{eq:smooth_process} implies\footnote{It is not hard to see that every $f_i$ has a bounded MGF $\pi$-almost surely.}
\begin{align}\label{eq:pac_bayes_tensor}
\E_{\xi_1 , \dots, \xi_n}\left[ \sup_{\rho_\theta \in \mc{M}_p } \frac{1}{\lambda}\sum_{k=1}^n \left( \E_{\theta \sim \rho_\theta}f_k(\xi_k,\theta) 
 -\E_{\theta \sim \rho_\theta}\log\left( \E_{\xi_k} \exp(f_k(\xi_k,\theta))\right) \right) - \frac{\KL(\rho_\theta \| \pi)}{\lambda} \right] \le 0.
\end{align}
Notice that we have restricted our attention to the set of probability measure $\mc{M}_p$.
This is crucial for our goal of computing explicit bounds on \cref{eq:tensor_process}.
All that remains is to bound both the logarithm of the MGF terms and the KL divergence term \emph{uniformly} for all $\rho_\theta \in \mc{M}_p$. 
We begin with the former.
Using the sub-gaussianity of $\xi_k$ (conditioned on $\theta \sim \rho_\theta$) followed by \cref{prop:second_moment_expanision} gives us, for any fixed $\rho_\theta$, that
\begin{align}
\sum_{k=1}^n\E_{\theta \sim \rho_\theta}\log\left( \E_{\xi_k} \exp(f_k(\xi_k,\theta))\right)
&\le \sum_{k=1}^n\frac{\lambda^2}{2} \E_{\theta \sim \rho_\theta}\left(\sum_{i_1 \in [d_1], \dots, i_r \in [d_r]} (T_k)_{i_1, \dots, i_r}\prod_{t=1}^r (\theta_t)_{i_t}\right)^2 \nonumber \\
&= \frac{\lambda^2}{2}\sum_{\ell=0}^r \beta^{-r+\ell} \sum_{k=1}^n 
 \sum_{\substack{I \subseteq [r]\\ |I| = \ell}} \|T_k^{\pi_I}(\theta_{\pi_I(1)}, \dots, \theta_{\pi_I(\ell)}, \cdot, \dots, \cdot )\|_{F}^2 \nonumber \\
&\le \frac{\lambda^2}{2}\sum_{\ell=0}^r \beta^{-r+\ell} \sup_{\theta_1 \in B_p^{d_1}, \dots, \theta_r \in B_p^{d_r}}\sum_{k=1}^n\sum_{\substack{I \subseteq [r] \nonumber \\ |I| = \ell}} \|T_k^{\pi_I}(\theta_{\pi_I(1)}, \dots, \theta_{\pi_I(\ell)}, \cdot, \dots, \cdot )\|_{F}^2 \nonumber \\
&= \frac{\lambda^2}{2}\sum_{\ell=0}^r \beta^{-r+\ell} \; \astd{r-\ell}^2= \frac{\lambda^2}{2}\sum_{\ell=0}^r \beta^{-\ell} \; \astd{\ell}^2.\label{eq:mgf_bound}
\end{align}
We now bound the KL divergence term. 
From the standard equation for the KL divergence between two univariate Gaussians together with the tensorization of KL for product measures we can conclude the bound
\begin{equation}\label{eq:kl_bound}
\KL(\rho_\theta \| \pi) = \frac{\beta \sum_{t=1}^r\|\theta_t\|_2^2}{2} \le\frac{\beta \sum_{t=1}^r d_t^{1 - \frac{2}{p}}}{2}.
\end{equation}
We can use the uniform bounds in \cref{eq:mgf_bound,eq:kl_bound} to further lower bound the left-hand side of \cref{eq:pac_bayes_tensor} which, after rearranging, yields
\begin{align*}
\E_{\xi_1, \dots, \xi_n}\left[ \sup_{\rho_\theta \in \mc{M}_p} \frac{1}{\lambda}\sum_{k=1}^n \E_{\theta \sim \rho_\theta}f_k(\xi_k,\theta) \right] &\le \frac{\lambda}{2}\sum_{\ell=0}^r \beta^{-\ell} \; \astd{\ell}^2 + \frac{\beta \sum_{t=1}^r d_t^{1 - \frac{2}{p}}}{2\lambda} \\
&= \sqrt{\left(\sum_{\ell =0}^r \beta^{1-\ell}\astd{\ell}^{2}\right)\sum_{t=1}^r d_t^{1 - \frac{2}{p}}}.
\end{align*}
The equality above follows from an explicit optimization of the parameter $\lambda > 0$.
Since $\beta > 0$ was arbitrary, the final line above together with \cref{eq:smooth_process} concludes the proof.
\end{proof}

\cref{cor:interpretable} is now a simple consequence of the following proposition. 
We defer the proof to~\cref{app:optimization}.
\begin{proposition}\label{prop:optimization}
For any $r \ge 2$ and $a_0, a_2, \cdots , a_r > 0$ we have
\[
2(r-1)^{\frac{1-r}{r}}\max_{2 \le \ell \le r} a_\ell^{\frac{1}{\ell}}a_0^{\frac{\ell-1}{\ell}} \le \inf_{x > 0 } \left\{ a_0 x + \sum_{\ell=2}^{r}  a_\ell x^{1-\ell} \right\} \le r \max_{2 \le \ell \le r} a_\ell^{\frac{1}{\ell}}a_0^{\frac{\ell-1}{\ell}}.
\]
\end{proposition}
\begin{proof}[Proof of \cref{cor:interpretable}]
Invoking \cref{thm:main} and \cref{prop:optimization} with $(a_\ell) = (\astd{\ell}^2)$ yields the claim.
\end{proof}

\section{Lata{\l}a's estimate on the moments of Gaussian chaoses}\label{sec:moments}
In this section we provide a short proof of~\cref{thm:moment_gaussian_chaos}.
We would like to emphasize once more that the proof follows Lata{\l}a's original argument, up to the use of \cref{thm:main}.
This difference drastically shortens and simplifies the \emph{overall} proof.
We will need the following corollary of \cref{thm:main}.
\begin{corollary}\label{cor:tensor_moments}
Let $T \in (\mathbb{R}^d )^{\otimes r-1}$ be a subgaussian random tensor.
For any partition $\mc{P} \in S(r-1)$ and $p \ge 2$ we have
\[
\left(\E \normtwo{{T^{\mc{P}}}}^p\right)^{\frac{1}{p}} \le C_{|\mc{P}|}\sum_{\mc{Q} \in S(r)} p^{\frac{-|\mc{P}|+|\mc{Q}|}{2}}\normtwo{ A(\{T_k\}_{k=1}^n)^{\mc{Q}}}.
\]
\end{corollary}
\begin{proof}
Let $m \le r-1$ be the cardinality of the partition $\mc{P} = \{I_1, \dots, I_m\} \in S(r-1)$.
For $i \in [m]$ define $d_i = d^{|I_i|}$ and notice that 
\[
\normtwo{T^{\mc{P}}} = \sup_{x_1 \in B_2^{d_1}, \dots, x_{m} \in B_2^{d_m}} \sum_{i_r=1}^d g_{i_r} A_{i_r}^{\mc{P}}(x_1, \dots, x_m)
\]
is the supremum of a (complicated) Gaussian process. 
Using the well known bound on the tails of the supremum of a gaussian process \cite[Lemma 2.10.6]{talagrand2022} together with the equivalence of tails and moments \cite[Exercise 2.3.8]{talagrand2022} gives us the following:
\begin{equation}\label{eq:moments_gaussian}
\left(\E \normtwo{T^{\mc{P}}}^p\right)^{\frac{1}{p}} \le \E\normtwo{T^{\mc{P}}} + C\sqrt{p}\sigma
\end{equation}
where the variance $\sigma^2$ is defined as 
\begin{align*}
\sigma^2 &= \sup_{x_1 \in B_2^{d_1}, \dots, x_m \in B_2^{d_m}} \E T^{\mc{P}}(x_1, \dots, x_m)^2 = \sup_{x_1 \in B_2^{d_1}, \dots, x_m \in B_2^{d_m}} \sum_{i_r=1}^{d} A_{i_r}^{\mc{P}}(x_1, \dots, x_m)^2 = \normtwo{A^{\mc{P}\cup\{\{r\}\}}}^2.
\end{align*}
Applying \cref{thm:main} with $\beta = p$  to the right-hand side of~\cref{eq:moments_gaussian}  and recalling  \cref{prop:partition_norm} yields
\[
\left(\E\normtwo{T^{\mc{P}}}^p\right)^{\frac{1}{p}} \le C_m \sum_{\mc{Q} \in S(r)}p^{\frac{-m+|\mc{Q}|}{2}}\normtwo{A^{\mc{Q}}} + C\sqrt{p}\normtwo{A^{\mc{P}\cup \{\{r\}\}}} \le C_m' \sum_{\mc{Q} \in S(r)}p^{\frac{-m+|\mc{Q}|}{2}} \normtwo{A^\mc{Q}},
\]
as desired.
\end{proof}
\begin{proof}[Proof of \cref{thm:moment_gaussian_chaos}]
We will prove the result via induction on the order $r$.
For $r=1$ standard bounds on the moments of a Gaussian random variable imply that
\[
\left\|\sum_{i =1}^d A_i g_i \right\|_p \le C \sqrt{p} \left(\sum_{i=1}^d A_i^2 \right)^\frac{1}{2} = C\sqrt{p}\normtwo{A^{\{\{1\}\}}}.
\]
Assume that the claim is true for chaoses of order $r-1$.
Let $T = A(\cdot, \dots, \cdot, g_r) \in (\bR^d)^{\otimes r-1}$.
We can write
\[
X \coloneqq \sum_{i_1, \dots, i_r \in [d]}A_{i_1, \dots, i_r} \prod_{t=1}^r (g_t)_{i_t} = \sum_{i_1, \dots i_{r-1} \in [d]} T_{i_1, \dots, i_{r-1}} \prod_{t=1}^{r-1} (g_t)_{i_t}
\]
and from our inductive hypothesis we have, conditioned on $g_r$, that
\[
\left(\E\left( |X|^p \mid g_r\right)\right)^{\frac{1}{p}} \le C_{r-1} \sum_{\mc{P} \in S(r-1)} p^{\frac{|\mc{P}|}{2}} \normtwo{T^{\mc{P}}}.
\]
Minkowski's (triangle) inequality for $L^p$ norms yields
\[
\| X \|_p \le C_{r-1} \left( \E \bigg(\sum_{\mc{P} \in S(r-1)} p^{\frac{|\mc{P}|}{2}} \normtwo{T^{\mc{P}}}\bigg)^p \right)^{\frac{1}{p}} \le C_{r-1} \sum_{\mc{P} \in S(r-1)} p^{\frac{|\mc{P}|}{2}} \left(\E\normtwo{T^{\mc{P}}}^p\right)^{\frac{1}{p}}.
\]
Applying~\cref{cor:tensor_moments} to the right-hand side of the above gives us
\[
\| X \|_p \le C_{r-1} \sum_{\mc{P} \in S(r-1)} p^{\frac{|\mc{P}|}{2}}C_{|\mc{P}|} \sum_{\mc{Q} \in S(r)}p^{\frac{-|\mc{P}| + |\mc{Q}|}{2}} \normtwo{A^{\mc{Q}}} \le  C_{r} \sum_{\mc{P} \in S(r)}  p^{\frac{|\mc{P}|}{2}} \normtwo{A^{\mc{P}}} ,
\] 
as desired.
\end{proof}
\paragraph{Acknowledgments.} 
I would like to thank Haotian Jiang and Kevin Lucca for valuable discussions on the connections between their work and the results presented here.
I am also very grateful to Rafa{\l} Lata{\l}a and David Wu for carefully reading an earlier version of this paper and for their helpful comments which improved the presentation.
Finally, I am indebted to Omar Alrabiah for allowing me to include his proof of~\cref{prop:optimization}, and for \emph{many} enjoyable discussions on the closely related problem of proving lower bounds for Locally Decodable Codes.
{\footnotesize
\bibliographystyle{alpha}
\bibliography{refs}
}

\appendix
\section{Proof of the PAC-Bayesian lemma}\label{app:pac_bayes_proof}
In this subsection we will prove \cref{lem:pacbayes}.
The proof is a direct consequence of the following result, known as the Donsker-Varadhan variational formula.
\begin{lemma}\label{lem:dv}
Fix a probability space $(\Theta,\pi)$.
For any measurable function $g$ such that $\E\limits_{\theta \sim \pi} \exp(g(\theta)) < \infty$ we have
\[
\log \left(\E_{\theta \sim \pi} \exp(g(\theta)) \right)= \sup_{\rho \ll \pi} \left\{\E_{\theta \sim \rho} g(\theta) - \KL(\rho \| \mu) \right\}.
\]
\end{lemma}
\begin{proof}
Let $\pi_g$ be the probability measure with density $\frac{\mathrm{d} \pi_g}{\mathrm{d} \pi}(\theta) = \frac{\exp(g(\theta))}{\E_{\theta \sim \pi} \exp(g(\theta))}.$
For any $\rho \ll \pi$ we have
\[
\KL(\rho \| \pi_g) = \int_{\Theta} \log \left(\frac{\mathrm{d} \rho}{\mathrm{d} \pi_g}\right)\mathrm{d}\rho = \KL(\rho \| \pi) +\log \left(\E_{\theta \sim \pi_g} \exp(g(\theta))\right) - \E_{\theta \sim \rho} g(\theta).
\]
Using the non-negativity of the KL divergence yields
\begin{equation}\label{eq:donsker_kl_lb}
\E_{\theta \sim \rho} g(\theta)-\KL(\rho \| \pi) \le \log \left(\E_{\theta \sim \pi} \exp(g(\theta))\right).
\end{equation}
Noticing that $\rho \ll \pi$ was chosen arbitrarily together with the fact that \cref{eq:donsker_kl_lb} is an equality when $\rho = \pi_g$ yields the claim.
\end{proof}
We are now ready to prove the PAC-Bayesian lemma.
Define the function $g : \mc{Z}^n \times \theta \to \mathbb{R}$ as 
\[g(Z_1, \dots, Z_n, \theta) = \sum_{k=1}^n \left(f_k(Z_k, \theta) -\log\left(\E_{Z_k} \exp(f(Z_k, \theta))\right)\right).
\]
Using \cref{lem:dv}(conditioned on $Z_1, \dots, Z_n$), Jensen's inequality, and the independence of $Z_1, \dots, Z_n$ yields
\begin{align*}
\E_{Z_1, \dots, Z_n}\left[\sup_{\rho \ll \pi}\E_{\theta \sim \rho }g(Z_1, \dots, Z_n, \theta) - \KL(\rho \| \pi)\right] &\le   \log \left(\E_{Z_1, \dots, Z_n}\E_{\theta \sim \pi} \exp(g(Z_1, \dots, Z_n, \theta)) \right)\\
&= \log \left(\E_{\theta \sim \pi} \E_{Z_1, \dots, Z_n}\exp\left(\sum_{k=1}^n f_k(Z_k, \theta) - \log \E_{Z_k}\exp(f_k(Z_k, \theta))\right) \right)\\
&= \log \left(\E_{\theta \sim \pi}\frac{\E_{Z_1, \dots, Z_n}\exp\left(\sum_{k=1}^n f_k(Z_k, \theta)\right)}{\prod_{k=1}^n \E_{Z_k} \exp(f_k (Z_k, \theta))} \right) = \log(1) = 0,
\end{align*}
as desired.
\section[Proof of]{Proof of \cref{prop:optimization}}\label{app:optimization}
We begin with the lower bound.
Using the trivial inequality to lower bound the infimum gives us
\begin{align*}
\inf_{x > 0 } \left\{ a_0 x + \sum_{\ell=2}^{r}  a_\ell x^{1-\ell} \right\} \ge \sum_{\ell=2}^{r} \inf_{x_\ell > 0 } \left\{ \frac{a_0 x_\ell}{r-1} +  a_\ell x_\ell^{1-\ell} \right\} \ge 2\sum_{\ell=2}^r (r-1)^{\frac{1-\ell}{\ell}}a_0^{\frac{\ell-1}{\ell}} a_\ell^{\frac{1}{\ell}} \ge 
2(r-1)^{\frac{1-r}{r}}\max_{2 \le \ell \le r} a_0^{\frac{\ell-1}{\ell}} a_\ell^{\frac{1}{\ell}},
\end{align*}
where the second inequality follows from an explicit optimization and the final inequality follows from the fact that the function $\ell \to (r-1)^{\frac{1-\ell}{\ell}}$ is non-increasing for $\ell, r \ge 2$.

To prove the upper bound, we will relax our optimization problem.
We can upper bound our original optimization problem as follows:
\begin{equation}\label{eq:relaxation}
\inf_{x > 0 } a_0 x + \sum_{\ell=2}^{r}  a_\ell x^{1-\ell}  \le \inf_{x > 0 }  \max_{\ell \in \{0, 2, \dots, r\}} r a_\ell x^{1-\ell}.
\end{equation}
It will be convenient to apply a change of variables.
Let $y = \ln(x)$ and $b_\ell = \ln(a_\ell)$ for $\ell \in I$.
Define the functions $f(y) = b_0 + y$ and $g(y) = \max\limits_{ \ell \in I} b_\ell + (1-\ell)y$.
We can re-write the right-hand side of the above as
\begin{equation}\label{eq:relaxation_v2}
\inf_{x > 0 }  \max_{\ell \in \{0 , 2, \dots, r\}} r a_\ell x^{1-\ell} = \inf_{y \in \bR } \max_{\ell \in \{0 , 2, \dots, r\}} r\exp(b_\ell + (1-\ell)y) =  \inf_{y \in \bR} r \exp
\left( \max\left\{f(y), g(y) \right\}\right).
\end{equation}
We will prove that there exists a $y^\star \in \bR$ such that $g(y^\star) = f(y^\star)$, i.e.\ 
\begin{equation}\label{eq:optimal_y}
b_0 + y^\star = \max_{2 \le \ell \le r} b_\ell + (1-\ell)y^\star.
\end{equation}
Before we prove this, we will show that such a $y^\star$ suffices to prove our desired upper bound.
Let $\ell^\star \in \{2, \dots, r\}$ be any index achieving the maximum on the right-hand side of \cref{eq:optimal_y}. 
Solving for $y^\star$ in terms of $b_0, \ell^\star$ and $b_{\ell^\star}$ and plugging this back into \cref{eq:relaxation_v2} yields
\[
\inf_{y \in \bR} r \exp
\left(\max\left\{ f(y) , g(y)\right\}\right) \le r \exp
\left(\max\left\{f(y^\star), g(y^\star) \right\}\right) = r\exp\left(\left(\frac{\ell^\star-1}{\ell^\star}\right)b_0 + \frac{1}{\ell^\star}b_{\ell^\star}\right) .
\]
Plugging the above back into our original optimization problem via \cref{eq:relaxation} yields
\[
\inf_{x > 0 } a_0 x + \sum_{\ell=2}^{r}  a_\ell x^{1-\ell}  \le r\exp\left(\left(\frac{\ell^\star-1}{\ell^\star}\right)b_0 + \frac{1}{\ell^\star}b_{\ell^\star}\right)  = r a_{0}^{\frac{\ell^\star -1}{\ell^\star}}a_{\ell^\star}^{\frac{1}{\ell^\star}} \le r \max_{2 \le \ell \le r} a_{0}^{\frac{\ell -1 }{\ell}}a_{\ell}^{\frac{1}{\ell}},
\]
as claimed.

We will now prove that such a $y^\star \in \bR$ exists.
Notice that the functions $f(y)$ and $g(y)$ are continuous, so the function $h(y) = f(y) - g(y) $ is continuous.
Furthermore, for a large enough $y_1$, we have $f(y_1) > g(y_1)$ which implies $h(y_1) > 0$. 
For a small enough $y_2$, we have $f(y_2) < g(y_2)$ which implies $h(y_2) < 0$.
The intermediate value theorem implies that a $y^\star$ satisfying $h(y^\star) = 0$ exists.
This completes the proof.
\end{document}